\documentclass[a4paper,11pt]{amsart}
\usepackage{color}
\usepackage{graphicx}
\usepackage{amsmath}
\usepackage{amssymb}
\usepackage{hyperref}
\usepackage{mathtools}

\newtheorem{thm}{Theorem}%[section] 
\newtheorem{lem}[thm]{Lemma}%[section]
\newtheorem{prop}[thm]{Proposition}%[section]
\theoremstyle{definition} 
\newtheorem{defn}{Definition}[section]
\theoremstyle{remark}
\theoremstyle{plain}

\def\NN{{\mathbb N}}

\def\RR{{\mathbb R}}

\def\ZZ{{\mathbb Z}}

\def\scrU{{\mathcal U}}

\def\dim{\operatorname{dim}}

\def\SL{\operatorname{SL}}

\def\Stab{\operatorname{Stab}}

\def\Onder#1#2#3#4#5{#1 \setbox0=\hbox{$#1$}\setbox1=\hbox{$#2$}
       \dimen0=.5\wd0 \dimen1=\dimen0 \dimen2=\dp0 \dimen3=\dimen2
       \advance\dimen0 by .5\wd1 \advance\dimen0 by -#4
       \advance\dimen1 by -.5\wd1 \advance\dimen1 by -#4
       \advance\dimen2 by -#3 \advance\dimen2 by \ht1
       \advance\dimen2 by 0.3ex \advance\dimen3 by #5
        \kern-\dimen0\raisebox{-\dimen2}[0ex][\dimen3]{\box1}
       \kern\dimen1}
\newcommand{\R}{\mathbb{R}}
\newcommand{\Z}{\mathbb{Z}}

\newcommand{\ve}{\varepsilon}

\newcommand\defeq{:=}
\title{Simultaneous Diophantine approximation --- logarithmic  improvements}
\author{Alexander Gorodnik}
\author{Pankaj Vishe}
\begin{document}
\maketitle
\begin{abstract}
This paper is devoted to the study of a  problem of Cassels in multiplicative Diophantine approximation which involves minimising values of a product of 
affine linear forms computed at integral points. It was previously known that values
of this product become arbitrary close to zero, and we establish that,
in fact, they approximate zero with an explicit rate.
Our approach is based on investigating quantitative density of orbits
of higher-rank abelian groups.
	
%Given a positive function $h$ on $\ZZ$, a pair of real numbers $u,v$ is called 
%``$h$-generalized Littlewood type (GLT)'' if
%for all
%$\alpha,\beta\in \RR$ and for all $\gamma>0$ ,  $\liminf h(q)^{1-\gamma}q\langle
%q
%u-\alpha\rangle
%\langle qv-\beta\rangle $. We will prove that almost every pair $(u,v)$ 
%in $\RR^2$
%is $h$-$GLT$, where $h(q)=(\log_{(5)} q)^\delta$ for
%some $\delta>0 $.
\end{abstract}

\section{Introduction}

Let $\left< u \right>$ denote the distance of the real number $u$ to the nearest integer.
The sequence $\left< qu\right>$ with $q\in \mathbb{N}$ reflects how well $u$ is approximated by
rational numbers. In particular, it is well-known that for every $Q\ge 1$ one can find  $q\le Q$ such that
$\left< q u\right>\le 1/Q$, but there is a large set of numbers $u$ satisfying $\left< q u\right>\ge
c(u)/q$ for all $q$'s with some $c(u)>0$. 
%Such numbers are called {\em badly approximable}. 
The long-standing Littlewood conjecture concerns 
simultaneous approximation of
a pair of real numbers $u,v\in \mathbb{R}$. It asserts that
\begin{equation}\label{eq:lit1}
\liminf_{q\to\infty} q\left< q u \right>\left< q v \right>=0
\end{equation}
 holds for all $u,v\in\RR$.
This paper deals with the inhomogeneous version of this problem,
namely, whether the following relation
\begin{equation}\label{eq:lit2}
\liminf_{|q|\to\infty} |q|\left< q u -\alpha \right>\left< q v -\beta\right>=0
\end{equation}
holds for $u,v,\alpha,\beta\in\mathbb{R}$. % where $u,v$ are rationally independent. 
In this setting Cassels asked (see \cite[p.~307]{C}) whether there exists a pair $(u,v)$ for which the property (\ref{eq:lit2}) holds for all real numbers 
$\alpha,\beta$. 
This question was answered affirmatively by Shapira in \cite{SC}
who showed that this is true for almost all pairs $(u,v)$.
He also gave an explicit example of a family of 
algebraic numbers $(u,v)$ satisfying this property,
% namely, if $1,u,v$ form a rational basis of a totally real cubic 
%field. On the other hand, (\ref{eq:lit2}) 
and showed that it fails if $u$ and $v$ are rationally dependent.

It is natural to ask whether the Littlewood conjecture (\ref{eq:lit1})
and its inhomogeneous version
(\ref{eq:lit2}) admit quantitative improvements. It follows from the results of Gallagher \cite{Gal}
that for almost every $(u,v)\in\mathbb{R}^2$,
\begin{equation}\label{eq:qlit1}
\liminf_{q\to\infty} \;(\log q)^2 q\left< q u \right>\left< q v \right>=0.
\end{equation}
Peck \cite{Peck} showed if $1,u,v$ form a basis of a real cubic field, then
\begin{equation}\label{eq:qlit2}
\liminf_{q\to\infty} \;(\log q) q\left< q u \right>\left< q v \right><\infty.
\end{equation}
Pollington and Velani \cite{PV} proved that (\ref{eq:lit1}) holds with an additional $\log q$ factor for a large set of 
pairs $(u,v)$, and Badziahin and Velani \cite{BV} conjectured that
(\ref{eq:qlit2}) holds for all real numbers $u$ and $v$. 

Unlike in the {\em homogeneous} setting, literature on quantitative results in the {\em inhomogeneous} setting has been 
lacking. An old argument of 
Cassels readily implies that  for almost all $(u,v,\alpha,\beta)\in \R^4$,
$$
\liminf_{q\to\infty} (\log q)^2 q\left< q u -\alpha \right>\left< q v -\beta\right>=0
$$
(see, for instance, \cite[Theorem 3.3]{H}). 
The case with $\alpha=0$ was investigated by
Haynes, Jensen and Kristensen in \cite{HJK}.
They proved that for all badly aproximable $u$, and $v$ contained in 
a set of badly approximable numbers of full Hausdorff dimension depending on $u$,
$$
\liminf_{q\to\infty} (\log q)^{1/2-\epsilon} q\left< q u \right>\left< q v -\beta\right>=0 \quad\hbox{ with any } \epsilon>0
$$
holds for all $\beta$.
Setting $\alpha=0$ allowed in  \cite{HJK} to use tools developed in \cite{PV}, but
it seems unlikely that this approach could be applied 
when $\alpha$ is non-zero. 

Apart from these results, no other quantitative improvements
of the inhomogeneous property (\ref{eq:lit2}) are known to us. 
The aim of this paper is to establish the first
quantitative improvement of \eqref{eq:lit2}
with arbitrary $\alpha,\beta$.
In contrast with the existing analytical methods, dynamical ideas employed 
in this paper enable us to successfully deal with general
$\alpha,\beta$ at a cost of a weaker logarithmic saving.
The following theorem is a quantitative refinement of one of the main results from \cite{SC}.

\begin{thm} \label{th:main0}
There exists $\delta>0$ such that for almost all $(u,v)\in \mathbb{R}^2$,
$$
\liminf_{|q|\to\infty} \;(\log_{(5)} |q|)^\delta |q|\left< q u -\alpha\right>\left< q v -\beta\right>=0
$$
holds for all  $\alpha,\beta\in \mathbb{R}$. Here
  $\log_{(s)}$ denotes the $s$-th iterate of the function $\max (1,\log|x|)$.
\end{thm}

We note that our method, in principle, could also allow establishing this result for specific
pairs $(u,v)$ provided that corresponding orbits satisfy a certain quantitative recurrence property. 

In a subsequent paper \cite{GV}, we also extend Theorem \ref{th:main0} to the $p$-adic setting
motivated by the $p$-adic version of the Littlewood conjecture proposed by
de Mathan and Teuli\'e \cite{MT}.

\vspace{0.2cm}

The setting of Theorem \ref{th:main0} can be considered as a particular case of
a general problem of multiplicative Diophantine approximation for
affine lattices (also called grids) in the Euclidean space $\RR^d$.
A grid in $\RR^d$ is a subset of the form
$$
\ZZ x_1+\cdots +\ZZ x_d +w,
$$
where $x_1,\ldots, x_d\in \RR^d$ are linearly independent and $w\in \RR^d$.
To formulate this problem explicitly, we set $N(v):=v_1v_2\cdots v_d$ for a vector $v={}^t (v_1,\ldots,v_d)$ in $\RR^d$.

\begin{defn}
	Let $\Lambda$ be a grid in $\RR^d$ and $h:\RR^+\to (0,1)$ a function such that $h(x)\to \infty$ as $x\to \infty$.
	\begin{enumerate}
		\item[(i)]  We say that $\Lambda$ is \emph{multiplicatively approximable} if $0$ is a non-trivial accumulation 
		point of a sequence $N(v_n)$ with $v_n\in \Lambda$.
		\item[(ii)] We say that $\Lambda$ is \emph{$h$-multiplicatively approximable} if there exists a sequence $v_n\in\Lambda$ 
		such that $v_n\to \infty$ and $0<|N(v_n)|<h(\|v_n\|)^{-1}$.
	\end{enumerate}
\end{defn}

We note that this provides a natural generalisation of property \eqref{eq:lit2}. Indeed, for $u,v,\alpha,\beta\in \RR$, we consider the grid
\begin{equation}
\label{eq:uv}
\Lambda(u,v,\alpha,\beta)\defeq\{{}^t(x,xu-y-\alpha,xv-z-\beta):\, x,y,z\in \ZZ\}.
\end{equation}
It is easy to check that if the grid $\Lambda(u,v,\alpha,\beta)$ is multiplicatively approximable, then \eqref{eq:lit2} holds.
Moreover, assuming that the function $h$ is non-decreasing, 
if the grid $\Lambda(u,v,\alpha,\beta)$ is $h$-multiplicatively approximable, then
$$
\liminf_{|q|\to\infty} h(c_1|q|-c_2) |q|\left< q u -\alpha \right>\left< q v -\beta\right>\le 1
$$
for some $c_1,c_2>0$.

It was also proved in \cite{SC} that for almost every lattice $\Lambda$ in $\R^d$, the grid $\Lambda+v$ is multiplicatively approximable for all $v\in \R^d$.
Here we establish a quantitative refinement of this result.

\begin{thm} \label{th:main}
There exists $\delta>0$ such that for almost every lattice $\Delta$ in $\R^d$, every grid $\Delta+w$, $w\in \R^d$, is $h$-multiplicatively approximable with 
$h(x)=(\log_{(5)} x)^\delta$.
\end{thm}

The paper is organised as follows. In the following section 
we set up required notation and give a dynamical reformulation of the problem, 
which reduces our investigation to the study of a quantitative recurrence property for orbits of a higher-rank abelian group $A$ acting on the space of grids in the Euclidean space.
However, it is not easy to establish this recurrence property directly,
so in Section 3, we first investigate quantitative recurrence in a smaller space ---
the space of lattices. In particular, it would be crucial in the proof to
establish recurrence to neighbourhoods of lattices with compact $A$-orbits. 
In Section 4, we discuss properties of compact orbits and relevant density results. 
Finally, in Section 5 we give a proof of the main theorems performing local analysis
in a neighbourhood of a grid whose corresponding lattice has compact $A$-orbit.

\subsection{Acknowledgements} The authors would like to thank S. Velani for suggesting the problem and for his 
encouragement during the work on the project. The first author 
was supported by ERC grant 239606, and
the second author was supported by EPSRC programme
grant EP/J018260/1.

\section{Preliminaries}

In this section we introduce some basic notation regarding dynamics on the space of grids in $\RR^d$ and give a dynamical reformulation of 
the above Diophantine approximation problem.
We also introduce a collection of roots subgroup that provides a convenient system
of local coordinates.

\subsection{Space of grids}
Let $G$ denote the group of unimodular affine transformations of $\RR^d$.
Let us set $G_0\defeq \SL(d,\RR)$ and $V\defeq \RR^d$. Then $G\simeq V \rtimes G_0$.
We also set $\Gamma_0 \defeq \SL(d,\ZZ)$ and $\Gamma \defeq \ZZ^d\rtimes \Gamma_0$.
Then $\Gamma_0$ is lattice in $G_0$, and 
$\Gamma$ is lattice in $G$. 
The space  $X\defeq G_0/\Gamma_0$ can be identified with the space
of unimodular lattices in $\RR^d$, and the space $Y \defeq G/\Gamma$ can be identified with the space 
of affine unimodular lattices, which are also called unimodular grids. 
For $x\in X$ we denote by $\Delta_x$ the corresponding lattice in $\RR^d$,
and for $y\in Y$, we denote by $\Lambda_y$ the corresponding grid. 
We denote by $\pi:Y\to X$ the natural factor map.
We observe that $\Lambda_y=\Delta_{\pi(y)}+w$ for some $w\in V$.
Moreover, $w$ can chosen to be uniformly bounded when $\pi(y)$ varies over bounded subsets of $X$.

%Let $\phi_1$ and 
%$\phi_2$ denote the canonical injection of $G_0$ and $V$ in $G$ respectively:
%$$\phi_1(g):=(g,\vecnull),\:\:\:\phi_2(\vecv):= (e,\vecv), \:\:\:\textrm{for any }g\in G_0, \vecv\in V,$$
%where $e$ denotes the identity matrix. We frequently abuse the notation and denote $\phi_1(g)$ by $g$ (similarly with 
%$\phi_2$). 

\subsection{Dynamical reformulation of the multiplicatively approximable property}
We show that the multiplicatively approximable property can be reformulated
 in terms of dynamics of the group 
 $$
 A:=\{a=\hbox{diag}(a_1,\ldots,a_d):\, a_i>0\}
 $$ acting 
 on the space $Y$. More specifically, we show that the grid $\Lambda_y$ is
 $h$-multiplicatively approximable if the orbit $Ay$ visits certain shrinking subsets $\mathcal{W}({\vartheta,\ve})$ of $Y$.
Given $\ve,\vartheta>0$, we introduce the following non-empty open subsets of $Y$
\begin{align*}
\mathcal{W}({\vartheta,\ve})&\defeq \{y\in Y:\, \exists v\in
\Lambda_y \hbox{ such that } \|v\|< \vartheta\hbox{ and } 0<|N(v)|<\ve \}.
\end{align*}

We also denote by $\|\cdot\|$ the maximum norm on $\hbox{Mat}(d,\RR)$, and 
for a subset $S$ of $\hbox{Mat}(d,\RR)$, we set 
$$
S(T):=\{s\in S:\,\|s\|<T \}.
$$

\begin{prop}\label{p:dyn}
	Let $h$ be a nondecreasing function such that $h(x)\to \infty$ as $x\to\infty$.
	Suppose that  for $y\in Y$,
	\begin{equation}
	\exists\,  T_n\to \infty:\;\; A(T_n)y \cap \mathcal{W}(T_n,h(T_n^{d})^{-1})\ne \emptyset.
	 \tag{WR}
	\end{equation}
	Then the grid $\Lambda_y$ is $h$-multiplicatively appoximable.
\end{prop}

\begin{proof}
It follows from our assumption that there exist sequences $a^{(n)}\in A(T_n)$
and $v^{(n)}\in \Lambda_y$ such that 
$$
|a_i^{(n)} v_i^{(n)}|<T_n\quad\hbox{ for all $i$},
$$
and 
$$
|N(a^{(n)}v^{(n)})|=|N(v^{(n)})|\in (0,h(T_n)^{-1}).
$$
This, in particular, implies that $0\ne N(v^{(n)})\to 0$,
so that $v^{(n)}\to \infty$.
We deduce from the first inequality that
$$
|v_i^{(n)}|<\left(a_i^{(n)}\right)^{-1} T_n=\left(\prod_{j\ne i}a_j^{(n)}\right) T_n\le T_n^d.
$$
Hence, $\|v^{(n)}\|\le T_n^d$, and
since $h$ is nondecreasing, we conclude that  	
$$
0<N(v^{(n)})<h(\|v^{(n)}\|)^{-1}.
$$
This proves that the grid $\Lambda_y$ is
$h$-multiplicatively appoximable.
\end{proof}

Proposition \ref{p:dyn} reduces study of the problem of multiplicative approximation
to analysing property (WR) --- the quantitative recurrence property of $A$-orbits
with respect to the sets $\mathcal{W}(\vartheta,\ve)$ in $Y$.

\subsection{Root subgroups}\label{sec:roots}
The crucial ingredient in understanding dynamics of the $A$-action
on the spaces $X$ and $Y$ are the root subgroups, which we now introduce.
The adjoint action of $A$ on the Lie algebra of $G$ is diagonalisable,
and we denote by $\Phi(G)$ the set of roots $A$ which is the set of 
non-trivial eigencharacters of $A$ appearing in this action.
For each $\alpha\in\Phi(G)$, there is a one-parameter root subgroup
$U_\alpha=\{u_\alpha(t)\}_{t\in\RR}\subset G$ such that
$$
au_\alpha(t)a^{-1}=u_\alpha(\alpha(a)t)\quad\hbox{for $a\in A$ and $t\in\RR$.} 
$$
More explicitly, the set of roots consists of 
$$
\hbox{$\alpha_{ij}(a)=a_ia_j^{-1}$ for 
$1\le i\ne j\le d$ and $\beta_i(a)=a_i$ for $1\le i\le d$. }
$$
The corresponding root subgroups are the groups of affine transformations
defined by 
$$
\hbox{$u_{ij}(t)u=u+ t u_j e_i$ and $v_i(t)u=u+t e_i$ }\quad\hbox{ for $u\in \RR^d,$}
$$
where $e_1,\ldots,e_d$ denotes the standard basis of $\RR^d$.
We denote the set of roots of the first type by $\Phi(G_0)$ and the set of roots 
of the second type by $\Phi(V)$.
With a suitable ordering, the product maps 
$$
A\times \prod_{\alpha\in \Phi(G_0)} \RR \to G_0:\; (a,t_\alpha:\alpha\in\Phi(G_0))\mapsto a\left(\prod_{\alpha\in\Phi(G_0)}u_\alpha(t_\alpha)\right),
$$
$$
\prod_{\alpha\in \Phi(V)} \RR \to V:\; (t_\alpha:\alpha\in\Phi(G_0))\mapsto \prod_{\alpha\in\Phi(V)}u_\alpha(t_\alpha),
$$
and 
$$
A\times \prod_{\alpha\in \Phi(G)} \RR \to G:\; (a,t_\alpha:\alpha\in\Phi(G))\to \left(\prod_{\alpha\in\Phi(V)}u_\alpha(t_\alpha)\right) a\left(\prod_{\alpha\in\Phi(G_0)}u_\alpha(t_\alpha)\right)
$$
are diffeomorphisms in neighbourhoods of the origins.
We set
\begin{align}
\mathcal{U}_{G_0}(\ve)&\defeq \{a\in A:\, \|a-e\|<\ve \}\cdot \prod_{\alpha\in\Phi(G_0) }\{u_\alpha(t_\alpha):|t_\alpha|< \ve\}, \nonumber\\
\mathcal{U}_V(\ve)&\defeq \prod_{\alpha\in\Phi(V) }\{u_\alpha(t_\alpha):|t_\alpha|< \ve\}, \label{eq:u} \\ \mathcal{U}_G(\ve)&\defeq \mathcal{U}_V(\ve)\,\mathcal{U}_{G_0}(\ve).\nonumber
\end{align}
Then $\mathcal{U}_{G_0}(\ve)$, $\mathcal{U}_{V}(\ve)$, and $\mathcal{U}_{G}(\ve)$
define neighbourhoods of identity in the groups $G_0$, $V$, and $G$ respectively.
We also consider the neighbourhoods of identity
\begin{align}
\mathcal{O}_{G_0}(\ve)&\defeq \{g\in G_0:\,\, \|g-e\|<\ve\}, \nonumber\\
\mathcal{O}_{V}(\ve)&\defeq \{v\in V:\,\, \|v\|<\ve\}, \label{eq:O}\\
\mathcal{O}_{G}(\ve)&\defeq \{(v,g)\in G:\,\, \|v\|<\ve,\, \|g-e\|<\ve  \}. \nonumber
\end{align}
It is easy to check that there exists $c_0>0$ such that for every $\ve\in (0,1)$,
\begin{equation}
\label{eq:inclusion}
\mathcal{U}_{G_0}(\ve)\subset \mathcal{O}_{G_0}(c_0\,\ve),\quad
\mathcal{U}_{V}(\ve)\subset \mathcal{O}_{V}(c_0\,\ve),\quad
\mathcal{U}_{G}(\ve)\subset \mathcal{O}_{G}(c_0\,\ve).
\end{equation}

\vspace{0.2cm}

While establishing quantitative recurrence of $A$-orbits to the sets $\mathcal{W}(\vartheta, \ve)$ is the crux of the proof of our main results,
it turns out that analogous recurrence property is easy to verify for 
the root subgroups. In fact, as an intermediate step in the proof, we will have to establish recurrence to smaller sets which are defined as
\begin{align*}
\mathcal{W}({\vartheta,\ve_1, \ve_2})&\defeq \{y\in Y:\, \exists v\in
\Lambda_y \hbox{ such that } \|v\|< \vartheta\hbox{ and } \ve_1 <|N(v)|<\ve_2 \}
\end{align*}
for $\vartheta>0$ and $0<\ve_1<\ve_2$.

\begin{lem}
\label{l:quant}
Let $\alpha\in \Phi(G)$.
For every $\ve_1,\ve_2\in (0,1)$, $\ve_1<\ve_2$, and $y\in Y$,
there exist positive $\vartheta=O_{\pi(y)}(1)$, positive $t_+=O_{\pi(y)}(1)$, and negative $t_-=O_{\pi(y)}(1)$ 
such that 
$$
u_\alpha(t_+)y\in \mathcal{W}(\vartheta,\ve_1,\ve_2)\quad\hbox{and}\quad u_\alpha(t_-)y\in \mathcal{W}(\vartheta,\ve_1,\ve_2).
$$
\end{lem}

\begin{proof}
We first note that the grid $\Lambda_y$ can be written as
$\Lambda_y=\Delta_{\pi(y)}+w$, where $w$ belongs to a fixed fundamental domain for the lattice $\Delta_{\pi(y)}$. In particular, $\|w\|=O_{\pi(y)}(1)$.

Let us show that $v_i(t)y\in \mathcal{W}(\vartheta,\ve_1,\ve_2)$ for some positive $\vartheta=O_{\pi(y)}(1)$ and some positive $t=O_{\pi(y)}(1)$.
Using that $\|w\|=O_{\pi(y)}(1)$ and adding a suitable vector
from the lattice $\Delta_{\pi(x)}$, one can show that there exists a vector $z\in \Lambda_y=w+\Delta_{\pi(y)}$ such that 
$$
\|z\|=O_{\pi(y)}(1),\quad\; z_i<0,\quad\; |z_k|\ge 1\;\;\hbox{ for all $k$.}
$$
Indeed, since $\Delta_{\pi(y)}$ is a lattice, there exists $s\in \Delta_{\pi(y)}$
satisfying $s_i<0$ and $s_k\ne 0$ for all $k$. Then we can choose $z$ of the form
$z=w+\ell s$ with a suitable $\ell \in \NN$.
We have to choose $t$ so that the inequalities
$$
\ve_1<|N(v_i(t)z)|<\ve_2
$$
hold.
Since $N(v_i(t)z)=N(z)+t N_i(z)$ where $N_i(z):=\prod_{j\ne i} z_i$, 
these inequalities are equivalent to 
$$
\ve_1|N_i(z)|^{-1}<|z_i+t|<\ve_2|N_i(z)|^{-1}.
$$
Hence, we can take $t$ from the interval $(\ve_1|N_i(z)|^{-1}-z_i, \ve_2|N_i(z)|^{-1}-z_i)$.
Due to our choice of $z$, we have $t>0$ and $t=O_{\pi(y)}(1)$.
Also, $\|v_i(t)z\|=O_{\pi(y)}(1)$.
Hence, it follows that $v_i(t)y\in \mathcal{W}(\vartheta,\ve_1,\ve_2)$ with some $\vartheta=O_{\pi(y)}(1)$ as required.
Similarly, one can also show that there exists negative $t$ satisfying $v_i(t)y\in \mathcal{W}(\vartheta,\ve_1,\ve_2)$.

The proof that $u_{ij}(t)y\in \mathcal{W}(\vartheta,\ve_1,\ve_2)$ for 
some positive $\vartheta=O_{\pi(y)}(1)$ and some positive $t=O_{\pi(y)}(1)$ follows similar lines.
Since $\|w\|=O_{\pi(y)}(1)$, we can add to $w$ a vector from the lattice $\Delta_{\pi(x)}$
to show existence of $z\in \Lambda_y=w+\Delta_{\pi(x)}$ satisfying
$$
\|z\|=O_{\pi(y)}(1),\quad\; z_i>0,\quad\; z_j<0,\quad\; |z_k|\ge 1\;\;
\hbox{ for all $k$.}
$$
Since $N(u_{ij}(t)z)=N(z)+t z_j N_i(z)$, the inequalities
$$
\ve_1<|N(u_{ij}(t)z)|<\ve_2
$$
are equivalent to 
$$
\ve_1|N_i(z)|^{-1}|z_j|^{-1}<|z_iz_j^{-1}+t|<\ve_2|N_i(z)|^{-1} |z_j|^{-1},
$$
so that we can take $t$ from the interval 
$(\ve_1|N_i(z)|^{-1}|z_j|^{-1}-z_iz_j^{-1}, \ve_2|N_i(z)|^{-1} |z_j|^{-1}-z_iz_j^{-1})$.
Then $t>0$ and $t=O_{\pi(y)}(1)$.
Also, it is clear that $\|u_{ij}(t)z\|=O_{\pi(y)}(1)$.
Hence, it follows that $u_{ij}(t)y\in \mathcal{W}(\vartheta,\ve_1,\ve_2)$ with some $\vartheta=O_{\pi(y)}(1)$.
The argument with negative $t$ is similar.
\end{proof}

%Ultimately we reduce the proof of property (WR) to Lemma \ref{quant}.

\section{Quantitative recurrence estimates}
\label{sec:rec}

Quantitative recurrence plays an important role in the theory of Diophantine approximation. In particular, this connection was realised in 
Sullivan's work \cite{SU} and its subsequent generalization
\cite{KM} by Kleinbock and Margulis. While these papers deal with recurrence 
to shrinking cuspidal neighbourhoods, we have to investigate 
visits of $A$-orbits to shrinking neighbourhoods of specific points inside the space $X$.
The idea of our approach, which uses exponential mixing, is similar to \cite{KM},
but it will be essential to establish recurrence to neighbourhoods of particular shape
with respect to the root coordinate system introduced in Section \ref{sec:roots}.
Namely, we consider neighbourhoods of $x\in X$
defined by $\mathcal{U}_\ve (x)\defeq \mathcal{U}_{G_0} (\ve)x$, where 
$\mathcal{U}_{G_0}(\ve)$ is defined in \eqref{eq:u}.

The main goal of this section is to prove the following proposition.

\begin{prop} \label{p:rec1}
	Let $x_0\in X$ and $a_t$ be a non-trivial one-parameter subgroup of $A$.
	Then there exists a constant $\beta>0$, such that for almost every $x\in X$ and every $T>T_0(x)$, 
	$$
	a_t x\in \mathcal{U}_{T^{-\beta}}(x_0)\backslash \mathcal{U}_{T^{-\beta}/2}(x_0)\quad\hbox{ for some $t\in [0,T]$.}
	$$
	
\end{prop}

We denote by $\mu$ the normalised invariant measure on the space $X$
and consider a family of averaging operators
$$
A_T:L^2(X)\to L^2(X): f\mapsto \frac{1}{T}\int_0^{T} f(a_tx)dt.
$$
We begin by proving an $L^2$-estimate for the operators $A_T$.

\begin{lem} \label{l: l2 bound}
There exists $\alpha>0$ such that for every $T\ge 1$ and $f\in C^\infty_c(X)$, 
$$\left\|A_T(f)-\int_X f\,d\mu \right\|_{2}\ll T^{-\alpha}S(f),$$
where $S(f)$ denotes a suitable Sobolev norm.
\end{lem}

\begin{proof} 
	We recall the exponential mixing property (see, for instance, \cite[Sec.~3]{KS}): there exists $\delta>0$ such that for every $f_1,f_2\in C_c^\infty(X)$,
	\begin{equation}\label{exp_mix}
	\int_X f_1(a_t x) f_2(x)\, d\mu(x)= \left(\int_X f_1\, d\mu\right) \left(\int_X f_2\, d\mu\right)+O\left(e^{-\delta|t|}S(f_1)S(f_2) \right).
	\end{equation}
	This property will be used to establish the required $L^2$-bound.
	Without loss of generality, we can assume that $\int_X f\, d\mu=0$. Then
	using \eqref{exp_mix}, we deduce that for every $M>0$,
	\begin{align*}
	\|A_T(f)\|^2_2&=T^{-2}\int_{(t,s)\in [0,T]^2}\int_X
	f(a_tx)\overline{f}(a_sx)\,d\mu(x)\,dt\,ds\\
	&=T^{-2}\int_{(t,s)\in [0,T]^2}\int_X
	f(a_{t-s}x)\overline{f}(x)d\mu(x) \,dt\,ds\\
	&=T^{-2}\left(\int_{(t,s)\in [0,T]^2:|t-s|<M}+\int_{(t,s)\in [0,T]^2:|t-s|>M} \right)\\
	&\ll T^{-2}(T M\, \|f\|_2^2 +e^{-\delta M}T^2\, S(f)^2).
	\end{align*} 
	Now the lemma follows by choosing a suitable value of $M$.
\end{proof}

	We are now ready to apply a standard Borel-Cantelli type argument to prove Proposition \ref{p:rec1}.

\begin{proof}[Proof of Proposition \ref{p:rec1}]
	Let $\beta\in (0,1)$, which value will be specified later, and
	$$
	\Omega_T :=\{x\in X:\, a_tx \notin \mathcal{U}_{T^{-\beta}/2^\beta}(x_0)\backslash \mathcal{U}_{ T^{-\beta}/2}(x_0)\;\;\hbox{for all $0\leq t\leq T$}\}. 
	$$
	We recall from the previous section that the neighbourhoods $\mathcal{U}_{G_0}(\ve)$ are $\ve$-cubes with respect to a suitable smooth coordinate system,
	so that
	we can choose a non-negative compactly supported function $f_T$ such that 
	$$
	\hbox{supp}(f_T)\subset \mathcal{U}_{T^{-\beta}/2^\beta}(x_0)\backslash \mathcal{U}_{ T^{-\beta}/2}(x_0),\quad\quad \int_X f_T\, d\mu=1,\quad\quad S(f_T)\ll T^{c\beta}
	$$ with some fixed $c>0$, determined by the Sobolev norm.
	We observe that for $x\in \Omega_T$,  $A_T(f_T)(x)=0$, so that
	$$
	\int_{\Omega_T} \left|A_T(f_T)-\int_X f_T\, d\mu\right|^2\,d\mu=|\Omega_T|. 
	$$
	On the other hand, by Lemma \ref{l: l2 bound},
	$$
	\int_{\Omega_T} \left|A_T(f_T)-\int_X f_T\, d\mu\right|^2\,d\mu\le 
	\left\|A_T(f_T)-\int_X f_T\,d\mu \right\|^2_{2}\ll (T^{-\alpha}S(f_T))^2\ll T^{2c\beta-2\alpha}.
	$$
	We pick $\beta\in (0,1)$ sufficiently small to make the last exponent negative.
	Then the above estimates imply that 
	$$
	|\Omega_T|\ll T^{-\epsilon}
	$$ with some $\epsilon>0$. Hence, it follows from the Borel-Cantelli lemma that 
	the limsup of the sets $\Omega_{2^{k}}$ has measure zero. This means that for almost every $x\in X$, we have $x\notin \Omega_{2^{k}}$
	for all $k\ge k_0(x)$, i.e., for all sufficiently large $k$, there exists
	$t\in [0,2^{k}]$ such that 
	$$
	a_tx \in \mathcal{U}_{2^{-k\beta}/2^\beta}(x_0)\backslash \mathcal{U}_{ 2^{-k\beta}/2}(x_0).
	$$
	Finally, given general $T\ge 1$, we choose $k$ so that $2^{k}\le T<2^{k+1}$.
	Then $[0,2^{k}]\subset [0,T]$. Hence, for all sufficiently large $T$,
	there exists $t\in [0,T]$ such that 
	$$
	a_tx \in \mathcal{U}_{2^{-k\beta}/2^\beta}(x_0)\backslash \mathcal{U}_{2^{-k\beta}/2}(x_0)\subset \mathcal{U}_{T^{-\beta}}(x_0)\backslash \mathcal{U}_{T^{-\beta}/2}(x_0).
	$$
	This completes the proof.
\end{proof}

Proposition \ref{p:rec1} is sufficient for the proof of Theorem \ref{th:main}, but for the proof of Theorem \ref{th:main0} we need a more refined recurrence 
property. We consider the one-parameter subgroup $$
a_t:=\hbox{diag}(e^{-(d-1)t},e^{t},\ldots,e^{t}),
$$
and denote by $U$ the expanding horospherical subgroup of $G_0$ 
for $a_t$ defined by 
\begin{equation}
\label{eq:uuu}
U:=\{g\in G_0:\, a_t^{-1}ga_t\to e\hbox{ as } t\to\infty\}.
\end{equation}
We note that $U\simeq \RR^{d-1}$ and the group $U$ generated by the root subgroups $U_{21},\ldots, U_{d1}$. We prove a recurrence result for orbits staring from points in 
$Ux\subset X$.

\begin{prop} \label{p:rec2}
	Let $x_0,x\in X$. Then
	there exists a constant $\beta>0$, such that for almost every $u\in U$ and every $T>T_0(u)$, 
	$$
	a_t u x\in \mathcal{U}_{T^{-\beta}}(x_0)\backslash \mathcal{U}_{T^{-\beta}/2}(x_0)\quad
	\hbox{for some $t\in [0,T]$.}
	$$
\end{prop}

\begin{proof}
We note that it will be sufficient to prove Proposition \ref{p:rec2}
for almost all $u$ contained an open neighbourhood $U_0$ of identity in $U$.
Our first goal is to prove an analogue of Lemma \ref{l: l2 bound} for averages along $U_0x$.

We introduce a complementary to $U$ subgroup
$$
W:=\{g\in G_0:\, a_t ga_t^{-1} \hbox{ is bounded as $t\to\infty$}\}.
$$
The product map $W\times U\to G_0$ is a diffeomorphism in a neighbourhood
of identity. We fix a right-invariant Riemannian metric on $G_0$ which also defines a metric on $X=G_0/\Gamma_0$. Let $W_\sigma$ denote the open 
$\sigma$-neighbourhood of identity in $W$. We assume that $\sigma$ and 
$U_0$ are sufficiently small, so that the product map $W_\sigma\times U_0\to G_0$
is a diffeomorphism onto its image, and the projection map $g\mapsto gx$, 
$g\in W_\sigma U_0$, is one-to-one. Let $X_\sigma:=W_\sigma U_0 x \subset X$.
We note that the invariant measure on $W_\sigma U_0\subset G_0$ is the image
under the product map of a left invariant measure on $W_\sigma $ and a right invariant measure on $U_0$. After suitable normalisation, this measure projects to 
the measure $\mu$ on $X_\sigma$. This implies that
for every $f\in C_c^\infty(X)$,
\begin{align}\label{eq:e1}
\left\|A_T (f)(wux)-\int_X f\,d\mu \right\|_{L^2(W_\sigma\times U_0)}
&=
\left\|A_T (f)-\int_X f\,d\mu \right\|_{L^2(X_\sigma)}\\
&\le 
\left\|A_T (f)-\int_X f\,d\mu \right\|_{L^2(X)}
\ll T^{-\alpha} S(f), \nonumber
\end{align}
where in the last estimate we used Lemma \ref{l: l2 bound}.

We observe that for every $wux\in W_\sigma U_0x$ and every $t>0$,
\begin{align*}
d(a_t wux,a_tux)\le d(a_tw a_t^{-1},e)\ll\sigma.
\end{align*}
Hence, it follows from the Sobolev embedding theorem that
for a suitable Sobolev norm $S$,
\begin{align*}
|f(a_twux)-f(a_tu x) |\ll \sigma S(f),\quad f\in C_c^\infty(X).
\end{align*}
This also implies that $|A_T(f)(wux)-A_T(f)(ux)|\ll \sigma S(f)$, and 
\begin{align}\label{eq:e2}
\left\|A_T (f)(wux)- A_T(f)(ux) \right\|_{L^2(W_\sigma\times U_0)}\ll \sigma |W_\sigma|^{1/2} S(f).
\end{align}
Combining \eqref{eq:e1} and \eqref{eq:e2}, we deduce that
$$
\left\|A_T (f)(ux)-\int_X f\,d\mu \right\|_{L^2(W_\sigma\times U_0)}
\ll (T^{-\alpha}+\sigma |W_\sigma|^{1/2}) S(f).
$$
Hence,
\begin{align*}
\left\|A_T (f)(ux)-\int_X f\,d\mu \right\|_{L^2(U_0)}&=
|W_\sigma|^{-1/2} \left\|A_T (f)(ux)-\int_X f\,d\mu \right\|_{L^2(W_\sigma\times U_0)}\\
&\ll (|W_\sigma|^{-1/2} T^{-\alpha}+\sigma ) S(f)\\
&\ll (\sigma^{-\dim(W)/2} T^{-\alpha}+\sigma ) S(f),
\end{align*}
and taking $\sigma=T^{-\epsilon}$ for sufficiently small $\epsilon>0$,
we conclude that for all $T\ge 1$ and $f\in C_c^\infty(X)$,
$$
\left\|A_T (f)(ux)-\int_X f\,d\mu \right\|_{L^2(U_0)} \ll T^{-\alpha'}S(f)
$$
with some $\alpha'>0$. 

Finally, we note that the last estimate is a complete analogue of Lemma \ref{l: l2 bound} for averages along $U_0x$. Now we can apply exactly the 
same argument as in the proof of Proposition \ref{p:rec1} to conclude that almost every $u\in U_0$ satisfies the claim of the proposition.
\end{proof}

\section{Compact $A$-orbits}

Our argument is based on studying distribution of $A$-orbits in a neighbourhood of a compact $A$-orbit. This idea goes back to
the papers of Furstenberg \cite{Furst} and Berend \cite{B},
and in the context of Cartan actions it was developed by Lindenstrauss, Weiss \cite{LW} and Shapira \cite{SC}.
It would be sufficient for our purposes to know that there exists $x_0\in X$
with a compact $A$-orbit. In fact, it is known that every order in a totally real number field
gives rise to a compact $A$-orbit (see, for instance, \cite[Sec.6]{LW} for details).

From now on we fix $x_0\in X$ such that $Ax_0$ is compact.
Let $B:=\hbox{Stab}_A(x_0)$. It is a discrete cocompact subgroup of $A$.
The group $B$ acts on the fiber $\pi^{-1}(x_0)$
which can be naturally identified with the torus $\R^d/\Delta_{x_0}$,
where $\Delta_{x_0}$ denotes the lattice corresponding to $x_0$.
Every $y\in \pi^{-1}(x_0)$ corresponds to a grid
$\Lambda_y=\Delta_{x_0}+v$ with $v\in V$.
We say that $y=\pi^{-1}(x_0)$ is \emph{$q$-rational} if $qv\in \Delta_{x_0}$.
We note that $B$ preserves the set of $q$-rational elements 
which has cardinality $q^d$. 
Hence, if $y\in \pi^{-1}(x_0)$ is $q$-rational,
then the subgroup $B_1:=\hbox{Stab}_A(y)$ has finite index in $B$, namely $|B:B_1|\le q^d$.
In particular, this implies the following approximation property.

\begin{lem}\label{l:projection}
	There exists $c>0$ such that for every $a\in A$, one can choose $b\in B_1$ satisfying $$\|ab^{-1}\|\le \exp(c\, q^d).$$
\end{lem}

Our argument involves study of dynamics of the action for the groups  $B$ and $B_1$ in a neighbourhood of the fiber $\pi^{-1}(x_0)$. The crucial part will be 
played by two quantitative density results that we now state.
The first result (Theorem \ref{th:wang}), which was proved by Z.~Wang \cite{W}, establishes quantitative density in the fibers,
and the second result (Proposition \ref{p:dioph alpha}),
which is deduced from the Baker Theorem, will be used to prove density along orbits of root subgroups.

We say that $y\in \pi^{-1}(x_0)$ is Diophantine of exponent $k$ if 
$\Lambda_y=\Delta_{x_0}+v$ and for some $c>0$,
\begin{equation}
\label{eq:diophantine}
|qv-z|\ge c\, q^{-k+1}\quad
\hbox{for every $q\ge 2$ and $z\in \Delta_{x_0}$.}
\end{equation}
The following theorem allows to establish quantitative density
in fibers $\pi^{-1}(x_0)$ of the space $Y$ under a Diophantine condition.

\begin{thm}[Z. Wang \cite{W}]\label{th:wang}
	There exist $Q_0,\sigma>0$ and $c=c(x_0)>0$ such that for every $y\in \pi^{-1}(x_0)$ satisfying \eqref{eq:diophantine} and $Q\ge Q_0$, the set $B(Q^{k+2})y$ is $(\log_{(3)} Q)^{-\sigma}$-dense in 
the torus $\pi^{-1}(x_0)$.
\end{thm}

We note that this result is stated in \cite{W} (see \cite[Theorem~10]{W}) for the standard torus $\RR^d/\ZZ^d$ and balls defined by the Mahler measure, but it straightforward to extend it
to our setting.

\vspace{0.2cm}

On the other hand, if the point $y$ in the fiber is close to a $q$-rational point, we will analyse action of the group $B_1$ on orbits of the root subgroups and use the following proposition.

\begin{prop}\label{p:dioph alpha}
	There exists $\eta>1$ such that given $\alpha\in \Phi(G)$ and a subgroup $B_1 $ of $B$ of exponent $q$,
	for every $M\ge 1 $ and $t>0$, there exists $a\in B_1$ such that
	$$
	|\alpha(a)- t|\ll q t M^{-1} \quad\hbox{and}\quad \log \|a\|\ll |\log t| M^{\eta+1}.
	$$
\end{prop}

We note that $\eta$ is precisely the exponent appearing in the Baker estimate \eqref{eq:d2_0}.

In the proof of Proposition \ref{p:dioph alpha} we use the following lemma.

\begin{lem}\label{l:dioph eigenvalue}
	Let $S$ be a multiplicative subgroup of $\R^+$ generated by
	multiplicatively independent algebraic numbers $\lambda_1$ and $\lambda_2$. Then there exists $\eta>1$ such that
	for every $M\ge 1 $ and $t>0$, there exists $s=\lambda_1^{\ell_1}\lambda_2^{\ell_2}\in S$ satisfying
	$$
	|s- t|\ll t M^{-1}\quad\hbox{and}\quad |\ell_1|,|\ell_2|\ll |\log t| M^{\eta+1}.
	$$
	Moreover, if $S_1$ is an exponent $q$ subgroup of $S$, then
	there exists $s=\lambda_1^{\ell_1}\lambda_2^{\ell_2}\in S_1$ satisfying
	$$
	|s- t|\ll q t M^{-1}\quad\hbox{and}\quad |\ell_1|,|\ell_2|\ll  |\log t| M^{\eta+1}.
	$$
	
\end{lem}

\begin{proof}
	We set $a_1=\log \lambda_1$ and $a_2=\log \lambda_2$.
	By Minkowski's theorem, for every $M\ge 1$, there exists $(n_1,n_2)\in \Z^2\backslash \{(0,0)\}$ such that 
	\begin{equation}
	\label{eq:d1_0}
	|n_1 a_1+ n_2 a_2|\le M^{-1}\quad\hbox{and}\quad |n_1|,|n_2|\ll M.
	\end{equation}
	We set $a:=n_1a_1+n_2a_2$.
	We note that $a\ne 0$ because $\lambda_1$ and $\lambda_2$ are assumed to be multiplicatively independent. It is clear that we can arrange $a>0$. Let $b:=\lceil \frac{M^{-1}}{a}\rceil a$.
	Then
	\begin{equation}
	\label{eq:mmm}
	M^{-1}\le b< \left(\frac{M^{-1}}{a} +1\right)a\le 2M^{-1}.
	\end{equation}
	It follows from the Baker Theorem (see, for instance, \cite[Ch.~3]{B0}) that there exists $\eta>1$ such that for all
	$(m_1,m_2)\in \Z^2\backslash \{(0,0)\}$,
	\begin{equation}
	\label{eq:d2_0}
	|m_1 a_1+ m_2 a_2|\ge  \max(|m_1|,|m_2|)^{-\eta}.
	\end{equation} 
	Hence, we deduce from \eqref{eq:d1_0} and \eqref{eq:d2_0} that  $\lceil\frac{M^{-1}}{a}\rceil\ll M^{\eta-1}$, so that $b=\ell_1a_1+\ell_2 a_2$ with 
	$|\ell_1|,|\ell_2|\ll M^\eta$. It follows from \eqref{eq:mmm} that
	the set $\{ib: |i|\le LM\}$ forms a $2M^{-1}$-net of the interval $[-L,L]$. Hence, for every $t>0$, there exists $d=i\ell_1a_1+i\ell_2 a_2$
	such that 
	$$
	|d- \log t|\ll M^{-1} \quad\hbox{and}\quad |i\ell_1|,|i\ell_2|\ll |\log t| M^{\eta+1}.
	$$
	This implies that $|e^d-t|\ll t M^{-1}$, as required.
	
	To prove the second part of the lemma, we apply the above argument to the elements $\lambda_1^q$ and $\lambda_2^q$ that belong to the subgroup $S_1$. It follows from \eqref{eq:d1_0}
	that there exists $(n_1,n_2)\in \Z^2\backslash \{(0,0)\}$ such that
	$$
	|n_1 a_1+ n_2 a_2|\le q M^{-1}\quad\hbox{and}\quad |n_1|,|n_2|\ll M.
	$$
	We set $a:=n_1a_1+n_2a_2$ and $b:=\lceil \frac{qM^{-1}}{a}\rceil a$.
	Then it follows from \eqref{eq:d2_0} that $a\ge qM^{-\eta}$. We proceed exactly as in the previous paragraph to prove the second part of the lemma.
\end{proof}	

\begin{proof}[Proof of Proposition \ref{p:dioph alpha}]
We write $x_0=g_0\Gamma_0$ for some $g_0\in G_0$. Then $B=A\cap g_0\Gamma_0 g_0^{-1}$.
It follows that entries of elements in $B$ are eigenvalues of matrices
from $\Gamma_0=\hbox{SL}(d,\ZZ)$. Hence, entries of elements in $B$ are algebraic numbers.
In particular, the group $\alpha(B)$ consists of algebraic numbers.
We apply Lemma \ref{l:dioph eigenvalue} to this group.
It was proven in \cite[Cor.~3.3]{SC} that $\alpha(B)$ is dense in $\mathbb{R}^+$
for every $\alpha\in \Phi(G)$. Since $B$ is finitely generated, this implies
that $\alpha(B)$ must contain two multiplicatively independent elements.
Now Proposition \ref{p:dioph alpha} follows directly from Lemma \ref{l:dioph eigenvalue}.
\end{proof}

\section{Proof of the main theorems}
  
The proof of the main theorems will use the dynamical reformulation of 
the `multiplicatively approximable' property stated in Proposition \ref{p:dyn}.
More explicitly, we will establish that for points $y$ 
in the space of grids $Y$, their orbits $A(R)y$ visit the shrinking sets $\mathcal{W}(\vartheta, (\log_{(5)} R)^{-\zeta})$, provided that $R$ is sufficiently large. As the first step, we use  
the results from Section \ref{sec:rec} to deduce that the projected orbits
$A(R)x$,  with $x=\pi(y)$, in the space of lattices $X$ visit shrinking neighbourhoods of any
given point $x_0$ in $X$. We apply this observation when the point $x_0$ has compact $A$-orbit. This will allow to analyse behaviour of $A$-orbits locally in a neighbourhood of 
the fiber $\pi^{-1}(x_0)$.
The crucial step of the proof is the following proposition:

\begin{prop} \label{p:supp}
Let $x_0\in X$ such that $Ax_0$ is compact and $x\in X$.
We assume that for fixed $\nu,\beta>0$ and all sufficiently large $T$,
\begin{equation}
\label{eq:assump}
\exists a_0\in A:\quad \|a_0\|\le e^{\nu T}\quad\hbox{and}\quad a_0x\in\mathcal{U}_{T^{-\beta}}(x_0)\backslash \mathcal{U}_{T^{-\beta}/2}(x_0).
\end{equation}
Then there exist $\vartheta,\zeta>0$ such that for any $y\in\pi^{-1}(x)$ and
all sufficiently large $R$,
\begin{equation}\label{eq:int}
A(R)y \cap \mathcal{W}(\vartheta,(\log_{(5)}R)^{-\zeta})\ne \emptyset.
\end{equation}
\end{prop}

We begin by investigating how the recurrence property in Proposition \ref{p:supp} changes under small perturbations of the base point $y$. It will be convenient to consider the family of neighbourhoods of $y$ in $Y$ defined by $\mathcal{O}_\ve(y)\defeq \mathcal{O}_G(\ve)y$,
where $\mathcal{O}_G(\ve)$ is defined in \eqref{eq:O}.

\begin{lem}\label{l:quant1}
	Let $0<\ve<1$, $\ve^{1/2}\le\ve_1<\ve_2$, $\vartheta\ge 1$, and $a\in A(\ve^{-1/(2d)})$. 	Then for every $y\in Y$, if 
	\begin{equation}
	\label{eq:ass1}
	a y\in \mathcal{W}(\vartheta,\ve_1,\ve_2),
	\end{equation}
	then for all $y'\in \mathcal{O}_\ve (y)$,
	$$
	a y'\in \mathcal{W}(3\vartheta,c_1\ve_1,c_2\ve_2)
	$$
	with some $c_1,c_2>0$, depending only on $\vartheta$.
\end{lem}

\begin{proof}
Since $y'\in \mathcal{O}_\ve (y)$, we can write $y'=hy$ with $h\in \mathcal{O}_G(\ve)$.
The element $h$ can be written as $h=(v,g)$ with $v\in V$ satisfying $\|v\|<\ve$ and $g\in G_0$ satisfying $\|g-e\|< \ve$. Then $ay'=(aha^{-1}) ay$ where $aha^{-1}=(av, aga^{-1})$.
We observe that for $x\in\hbox{Mat}(d,\R)$, we have
$$
\|axa^{-1}\|\le \|a\|\cdot \|a^{-1}\|\cdot \|x\|\le \|a\|^d\cdot \|x\|,
$$
so that since $a\in A(\ve^{-1/(2d)})$, we deduce that
$$
\|aga^{-1} -e\|\le \ve^{-1/2} \|g-e\|<\ve^{1/2}.
$$
Also $\|av\|\le \|a\|\,\|v\|<\ve^{1/2}$.

According to our assumption \eqref{eq:ass1}, there exists $z\in \Lambda_{ay}$
such that $\|z\|<\vartheta$ and $\ve_1<N(z)<\ve_2$. 
Then the vector $w:=(aha^{-1}) z$ belongs to $\Lambda_{ay'}$, and
$$
w= (aga^{-1})z+av=z + ((aga^{-1})-e) z+av.  
$$
This implies that $\|w\|<3\vartheta$, and $w=z+O_\vartheta(\ve^{1/2})$,
so that $N(w)=N(z)+O_\vartheta(\ve^{1/2})$. Hence, $ay'\in \mathcal{W}(3\vartheta,c_1\ve_1,c_2\ve_2)$ for some $c_1,c_2>0$, depending only on $\vartheta$.
\end{proof}

\begin{proof}[Proof of Proposition \ref{p:supp}]
We write 
$a_0x=g_0x_0$
with $g_0\in \scrU_{G_0}(T^{-\beta})\backslash \scrU_{G_0}(T^{-\beta}/2)$.
The element $g_0$ has a decomposition
$$
g_0=c\prod_{\alpha\in\Phi(G_0)} u_\alpha(t_\alpha),
$$
where
$c\in A$ with $\|c-e\|<T^{-\beta}$, $|t_\alpha|< T^{-\beta}$ for all $\alpha$,
and $|t_{\alpha_0}|\geq T^{-\beta}/2$ for some $\alpha_0\in \Phi(G_0)$.
In particular, $g_0=e+O(T^{-\beta})$.
For every $y\in \pi^{-1}(x)$, 
 \begin{equation}
 \label{eq:start}
 a_0y=g_0 y_0
 \end{equation}
 with some $y_0\in \pi^{-1}(x_0)$.
The point $y_0$ corresponds to the grid $\Delta_{x_0}+g_0^{-1}w$ with some $w\in V$.
Since $g_0$ is bounded, $w$ can be chosen to lie in a fixed
bounded subset of $V$, depending only on $\Delta_{x_0}$. 
Although we don't have any control over $w$, 
we may assume (after modifying \eqref{eq:start}) that either
\begin{equation}
\label{eq:asss}
\|w\|\ge T^{-\beta}\quad\hbox{ or } \quad w=0.
\end{equation}
Indeed, suppose that $\|w\|< T^{-\beta}$. 
Then $\|g_0^{-1} w\|\ll T^{-\beta}$, so that
$g_0^{-1} w=\prod_{\alpha\in \Phi(V)} u_\alpha(t_\alpha)$ with $|t_\alpha|\ll T^{-\beta}$.
The element $g:=(g_0^{-1}w,g_0) \in G$ can be written as
\begin{equation}
\label{eq:g}
g=\left(\prod_{\alpha\in \Phi(V)} u_\alpha(t_\alpha)\right) c\left(\prod_{\alpha\in\Phi(G_0)} u_\alpha(t_\alpha)\right),
\end{equation}
and we can replace \eqref{eq:start} by the equation 
\begin{equation}
\label{eq:start2}
a_0y=g y_0',
\end{equation}
where the point $y_0'$
corresponds to the grid $\Delta_{x_0}$.  
On the other hand, if $\|w\|\ge T^{-\beta}$, we leave \eqref{eq:start} as it is.
We conclude that in any case we can obtain \eqref{eq:start2} with $g$ as in \eqref{eq:g},
where $y_0'$ corresponds to a grid $\Delta_{x_0}+g^{-1}w$ with $w$ satisfying
\eqref{eq:asss}. We note that $g=e+O(T^{-\beta})$, and $g$ has its $G_0$-component equal to $g_0$.

Since behaviour of the orbit $Ay$ depends crucially on 
the Diophantine properties of the vector $w$ with respect to the lattice $\Delta_{x_0}$, the proof naturally split into the following
three subcases:
\begin{enumerate}
	\item[1.] {\it $w$ is Diophantine:} for every $q\ge 2$ and $w_0\in \Delta_{x_0}$, $|qw-w_0|\ge c(x_0)\, q^{-k+1}$, where $c(x_0)$ is as in Theorem \ref{th:wang}.
	\item[2.] {\it $w$ is close to a torsion point with small period:} there exist $q\ge 2$ with $q\le L$ and $w_0\in \Delta_{x_0}$ such that $|qw-w_0|< c(x_0)\, q^{-k+1}$.
	\item[3.] {\it $w$ is close to a torsion point with large period:}
	for every $q\ge 2$ with $q\le L$ and $w_0\in \Delta_{x_0}$, $|qw-w_0|\ge c(x_0)\, q^{-k+1}$, but there exist $q\in \NN$ and $w_0\in \Delta_{x_0}$ such that $|qw-w_0|< c(x_0)\, q^{-k+1}$.
\end{enumerate}
The parameters $k$ and $L$ appearing above will be chosen of the form $k=k(T)\to \infty$ and $L=L(T)\to \infty$ as $T\to\infty$, and they will be specified in the course of the proof.

\vspace{0.1cm}

Now we investigate each of these cases separately:

\vspace{0.4cm}

\noindent \underline{Case 1:} {\it $w$ is Diophantine.}
It follows from Theorem \ref{th:wang} that 
for sufficiently large $Q$, the set $B(Q^{k+2})w$ is $(\log_{(3)} Q)^{-\sigma}$-dense in 
$V/\Delta_{x_0}$. We observe that $g^{-1}w=w+O(T^{-\beta})$.
Let us assume that $Q^{k+2}\le T^{\beta/2}$ (in fact, later in the proof we will have
to impose much stronger restriction).
Then the set $B(Q^{k+2})g^{-1}w$ is $2(\log_{(3)} Q)^{-\sigma}$-dense in 
$V/\Delta_{x_0}$, when $Q$ is sufficiently large. We observe that for $a\in B=\hbox{Stab}_A(x_0)$, 
$$
a\Lambda_{y_0}=\Lambda_{ay_0}=a(\Delta_{x_0}+g^{-1}w)=\Delta_{x_0}+ag^{-1}w.
$$
Hence, the set $B(Q^{k+2})\Lambda_{y_0}$ is also $2(\log_{(3)} Q)^{-\sigma}$-dense in $V$. In particular, this implies that 
there exist $a\in B(Q^{k+2})$ and $z\in \Lambda_{ay_0}$ such that 
$$
(\log_{(3)} Q)^{-\sigma}<z_i<3(\log_{(3)} Q)^{-\sigma}\quad\hbox{for every $i$.}
$$
These inequalities imply that 
$$
\|z\|< 1\quad\quad\hbox{and}\quad\quad
(\log_{(3)} Q)^{-d\sigma}<N(z)<3^d(\log_{(3)} Q)^{-d\sigma},
$$
so that
\begin{equation}\label{eq:WWW}
ay_0\in \mathcal{W}(1,(\log_{(3)} Q)^{-d\sigma}, 3^d(\log_{(3)} Q)^{-d\sigma})
\end{equation}
for all sufficiently large $Q$.

Next we claim that an analogous inclusion holds for the point $a gy_0$ as well.
Since $gy_0\in \mathcal{O}_{c_0\,T^{-\beta}}(y_0)$ for some fixed $c_0>0$ (see \eqref{eq:inclusion})
and $a\in B(Q^{k+2})$, we will be able to deduce this
from Lemma \ref{l:quant1} provided that $Q$ is not too large. 
Taking this into account, we choose $Q$ so that $Q^{k+2}=(c_0\,T^{\beta})^{1/(2d)}$.
We note that if $k=o(\log T)$ as $T\to\infty$, then $Q\to \infty$ as $T\to\infty$ and \eqref{eq:WWW} holds for all sufficiently large $T$.
Then Lemma \ref{l:quant1} implies that for some $c_1,c_2>0$,
$$
aa_0y=agy_0\in \mathcal{W}(3,c_1(\log_{(3)} Q)^{-d\sigma}, c_2(\log_{(3)} Q)^{-d\sigma}).
$$  
Since $Q^{k+2}=(c_0\,T^{\beta})^{1/(2d)}$,
$$
aa_0\in B(Q^{k+2})A(e^{\nu T})\subset A(e^{(\nu+1)T}).
$$
Hence, taking $R=e^{(\nu+1)T}$,  
we deduce that for all sufficiently large $R$,
\begin{equation}
\label{eq:RRR}
A(R)y\cap \mathcal{W}(3, f(R)^{-d\sigma})\ne \emptyset,
\end{equation}
where
\begin{align*}
f(R)\gg \log_{(3)} \left( c_0^{\frac{1}{2d(k+2)}} \left(\frac{\log R}{\nu+1}\right)^{\frac{\beta}{2d(k+2)}} \right)\gg
\log \left( \log_{(3)} R-\log k \right).
\end{align*}
In order for \eqref{eq:RRR} to give a non-trivial estimate, we have to choose $k$ such that
\begin{equation}
\label{eq:k1}
\log k= \log_{(3)} R -s(R)\quad\hbox{ with $s(R)\to \infty$.}
\end{equation}
We note that this choice of $k$, in particular, implies that $k=o(\log T)$, which we have used above. Finally, we deduce from \eqref{eq:RRR} that 
for all sufficiently large $R$,
\begin{equation}
\label{eq:RRR1}
A(R)y\cap \mathcal{W}(3, O((\log s(R))^{-d\sigma}))\ne \emptyset.
\end{equation}

\vspace{0.4cm}

\noindent \underline{Case 2:} {\it $w$ is close to a torsion point with small period.}
We start by modifying equation \eqref{eq:start2}. We observe that 
$$
g=c\left(\prod_{\alpha\in \Phi(V)} u_\alpha(t'_\alpha)\right) \left(\prod_{\alpha\in\Phi(G_0)} u_\alpha(t_\alpha)\right),
$$
where $t_\alpha'=\alpha(c)^{-1} t_\alpha=O(T^{-\beta})$.
Replacing $a_0$ by $a_0c^{-1}$, we may assume without loss of generality that
\eqref{eq:start2} holds with $c=e$.
We denote by $y_0''$ the element of $Y$ that corresponds to the grid $\Delta_{x_0}+q^{-1}w_0$.
Then 
\begin{equation}
\label{eq:first}
a_0y=g y'_0=  hy_0''\quad\hbox{ where $h:=(g^{-1}w-q^{-1}w_0,e)g\in G$. }
\end{equation}
We observe that 
$$
\|g^{-1}w-q^{-1}w_0\|\le \|g^{-1}w-w\|+\|w-q^{-1}w_0\|\ll \max( T^{-\beta}, q^{-k}).
$$
We decompose $h$ into a product with respect to root subgroups:
\begin{equation}
\label{eq:prod0}
h=\prod_{\alpha\in \Phi(G)} u_\alpha(t_\alpha),
\end{equation}
where
$|t_\alpha|\ll \max(T^{-\beta},q^{-k})$ for all roots $\alpha$.
We recall that $k$ is chosen so that $k=o(\log T)$ (see Case 1),
Hence, it follows that $|t_\alpha|\ll 2^{-k}$ for all roots $\alpha$, when $T$
is sufficiently large. 
Let $B_1:=\Stab_A(y_0'')$. Since $B_1$ is precisely the stabiliser in $B$ 
of the $q$-rational point $q^{-1}w_0$ in $V/\Delta_{x_0}$, it follows that
$|B:B_1|\le q^d$. We observe that for $b\in B_1$, we have 
\begin{equation}
\label{eq:conj}
bhy_0''= (bhb^{-1})y_0''= \left(\prod_{\alpha\in \Phi(G)} u_\alpha(\alpha(b) t_\alpha)\right) y_0''.
\end{equation}
Our argument is based on picking suitable elements $b\in B_1$ which 
contract some of the coordinates $t_\alpha$. This will allow to reduce complexity of the product in \eqref{eq:conj}. A useful tool for achieving this is the following elementary lemma.

\begin{lem}
	\label{l:linear}
	Let $v_1,\ldots, v_s$ be a collection of distinct vectors in a vector space $V$.
	Assume that there exists $L\in V^*$ such that $L(v_i)>0$ for all $i$.
	Then there exists $v_j$ such that 
	\begin{itemize}
	\item for some $S_1\in V^*$, $S_1(v_{j})>0$ and $S_1(v_{i})<0$ with $i\ne j$,
	\item for some $S_2\in V^*$, $S_2(v_{j})=0$ and $S_2(v_{i})<0$ with $i\ne j$.
	\end{itemize}
\end{lem}

\begin{proof}
Let $\bar v_i$ denote the positive multiple of $v_i$ such that $L(\bar v_i)=1$.
We denote by $\mathcal{C}$ the closed convex hull of the points $\bar v_i$. 
Let $\bar v_{j}$ be an extreme point of $\mathcal{C}$.
Then there exists a hyperplane $\mathcal{H}$ in $L=1$ which separates 
$v_j$ and $v_i$, $i\ne j$. It is sufficient to pick $S_1\in  V^*$ such that $\{S_1=0\}\cap \{L=1\}=\mathcal{H}$ with a suitable sign. The proof of the second part is similar.
\end{proof}	

We note that conjugating by elements from $B_1$, one can only achieve precision $e^{O(q^d)}$
(cf. Lemma \ref{l:projection}), but the coordinates $t_\alpha$ are of order $O(2^{-k})$,
so that our argument, which is presented below, could only work provided that $k\ge O(q^d)$. Since $q\le L$, 
it is sufficient to assume that the parameter
$L=L(T)$ satisfies
\begin{equation}
\label{eq:L1}
L^d=o(k)\quad \hbox{as $T\to \infty$.}
\end{equation}

In view of \eqref{eq:conj}, we have to analyse behaviour of 
$\alpha(b)t_\alpha$, $\alpha\in \Phi(G)$, with $b\in B_1$.
Now we construct an explicit $b\in B_1$ which contracts some of the factors in \eqref{eq:conj}.
Since $B$ is a lattice in $A$, there exists an element $a\in B$ such that $\alpha(a)\ne 1$
for all $\alpha\in\Phi(G)$. In particular, we have a decomposition 
$$
\Phi (G)=\Phi^+\sqcup \Phi^- \quad \hbox{ where $\Phi^+:=\{\alpha:\, \alpha(a)>1\}$ and $\Phi^-:=\{\alpha:\, \alpha(a)<1\}$.}
$$
This decomposition is non-trivial because $\prod_{\alpha} \alpha(a)=1$.
We recall that there exists $\alpha_0\in \Phi(G_0)$
such that $|t_{\alpha_0}|\ge T^{-\beta}/2$. 
Replacing $a$ by $a^{-1}$ if necessary we may assume that $\alpha_0\in \Phi^+$.
Let us pick the maximal exponent $i$ such that 
$\alpha(a)^{i|B:B_1|}|t_\alpha|\le 1$ for all  $\alpha\in \Phi(G)$
and set $b:=a^{i|B:B_1|}$. Clearly, $b\in B_1$.
Also since $\alpha_0(a)^{i|B:B_1|} T^{-\beta}/2\le 1$,
it follows that $i|B:B_1|\ll \log T$, so that 
$$
\|b\|\le T^{O(1)}.
$$
It follows from our choice of the exponent $i$ that there exists $\alpha_1\in \Phi(G)$ such that 
$$
\alpha_1(a)^{(i+1)|B:B_1|}|t_{\alpha_1}|= \alpha_1(b)\alpha_1(a)^{|B:B_1|}|t_{\alpha_1}| > 1.
$$
Hence, 
\begin{equation}
\label{eq:alpha1}
\alpha_1(b)|t_{\alpha_1}|\ge e^{-O(q^d)}.
\end{equation}
On the other hand, for all $\alpha\in\Phi^-$, we have $\alpha(b)|t_{\alpha}|<|t_{\alpha}|= O(2^{-k})$.
We conclude that
\begin{equation}
\label{eq:prod}
h_1:=bhb^{-1}=\prod_{\alpha\in \Phi(G)} u_\alpha(s_\alpha),
\end{equation}
where $|s_\alpha|\le 1$ for all $\alpha$, $|s_{\alpha_1}|\ge e^{-O(q^d)}$, and 
$|s_\alpha|\le  \omega\,2^{-k}$
for fixed $\omega>0$ and all $\alpha\in \Phi^-$.

Let us introduce a parameter $\zeta\in (0,1)$ which will be specified later (see \eqref{eq:zeta} below). 
Since the number of roots $\alpha$ is finite there exists $\ell\in \NN$ such that 
no coordinates $|s_\alpha|$ are contained in the interval $((\omega\,2^{-k})^{\zeta^{\ell-1}}, (\omega\,2^{-k})^{\zeta^{\ell}}]$.
We decompose the set of roots as $\Phi(G)=\Phi_1\sqcup \Phi_2$ where
$\Phi_1$ consists of $\alpha$ such that $|s_\alpha|\ge (\omega\,2^{-k})^{\zeta^{\ell}}$,
and $\Phi_2$ consists of $\alpha$ such that $|s_\alpha|\le (\omega\,2^{-k})^{\zeta^{\ell-1}}$.
We note that $\Phi^-\subset \Phi_2$.
Also, it follows from \eqref{eq:alpha1} and \eqref{eq:L1} that $\alpha_1\in \Phi_1$.
In particular, $\Phi_1$ is not empty.
We observe that for $\alpha\in \Phi_2$ and bounded $g\in G$, we have
$g u_\alpha(s_\alpha)=vg$ where $v=e+O(2^{-\zeta^{\ell-1} k})$. Therefore, we can rearrange the terms
in the product \eqref{eq:prod0}, so that 
\begin{equation}
\label{eq:prod}
h_1=v_1 h_2\quad\quad \hbox{where}\quad v_1=e+O(2^{-\zeta^{\ell-1} k})\;\;\hbox{and}\;\; h_2:=\prod_{\alpha\in \Phi_1} u_\alpha(s_\alpha).
\end{equation}
Now we apply Lemma \ref{l:linear} to the set $\Phi_1$ considered as a subset of the dual $A^*$ of $A$.
The condition of the lemma holds because $\Phi_1\subset \Phi^+$.
Hence, we deduce that there exist $a_1,a_2\in A=(A^*)^*$ and $\alpha_2\in \Phi_1$ such that
$\alpha_2(a_1)>1$ and $\alpha(a_1)<1$ for all $\alpha\in\Phi_1\backslash \{\alpha_2\}$,
and $\alpha_2(a_2)=1$ and $\alpha(a_2)<1$ for all $\alpha\in\Phi_1\backslash \{\alpha_2\}$.
Rescaling $a_1$, we arrange that $\alpha_2(a_1)=|s_{\alpha_2}|^{-1}$.
Since $|s_{\alpha_2}|^{-1}\ll 2^{\zeta^\ell k}$, there exists a constant $c_1>0$ such that
$\|a_1\|\le 2^{c_1 \zeta^\ell k}$. Moreover, $c_1$ depends only on the initial choice
of $a_1$, so that it is uniform. Furthermore, we also rescale $a_2$ so that 
$\|a_2\|\le 2^{c_1 \zeta^\ell k}$ and $\alpha(a_2)<2^{-\delta k}$ with some fixed $\delta>0$ for all $\alpha\in\Phi_1\backslash \{\alpha_2\}$. Then 
\begin{equation}
\label{eq:a12}
\|a_1a_2\|\le \|a_1\| \|a_2\|\le 2^{2c_1\zeta^\ell k},
\end{equation}
and 
\begin{equation}
\label{eq:a12_0}
\alpha_2(a_1a_2)|s_{\alpha_2}|=1, \quad \alpha(a_1a_2)|s_{\alpha}|< 2^{-\delta k} \hbox{ for all $\alpha\in\Phi_1\backslash \{\alpha_2\}$.}
\end{equation}
By Lemma \ref{l:projection}, there exists $b_2\in B_1$ such that 
\begin{equation}
\label{eq:abcd}
\|(a_1a_2)b_2^{-1}\|\le e^{O(q^d)}.
\end{equation}
It follows from \eqref{eq:a12} and \eqref{eq:L1} that
\begin{equation}
\label{eq:b_20}
\|b_2\|\le 2^{2c_1\zeta^\ell k} e^{O(q^d)}=2^{(2c_1\zeta^\ell+o(1)) k}.
\end{equation}
We have  
$$
h_3:=b_2 h_2 b_2^{-1} =\prod_{\alpha\in \Phi_1} u_\alpha(r_\alpha),
$$
where $r_\alpha =\alpha(b_2)s_\alpha$. By \eqref{eq:a12_0}, \eqref{eq:abcd} and \eqref{eq:L1}, 
$$
|r_{\alpha_2}|\ge e^{-O(q^d)}, \quad 
|r_{\alpha}|\le 2^{-\delta k}e^{O(q^d)}=2^{-(\delta-o(1)) k} \hbox{ for all $\alpha\in \Phi_1\backslash \{\alpha_2\}$.}
$$
Arguing as in \eqref{eq:prod}, we deduce that 
\begin{equation}
\label{eq:prod2}
h_3=v_2 u_{\alpha_2}(r_{\alpha_2})\quad\quad \hbox{where}\quad v_2=e+O(2^{-(\delta-o(1)) k}).
\end{equation}

Let us assume that $r_{\alpha_2}>0$ since the other case can be treated similarly.
By Lemma \ref{l:quant}, for every $0<\ve_1<\ve_2<1$, there exists positive $t_+=O_{\pi(y_0'')}(1)=O_{x_0}(1)$
such that 
\begin{equation}
\label{eq:eeee}
u_{\alpha_2}(t_+)y_0''\in \mathcal{W}(\vartheta,\ve_1,\ve_2),
\end{equation}
where $\vartheta=O_{\pi(y_0'')}(1)=O_{x_0}(1)$.
By Proposition \ref{p:dioph alpha},
for every $M\ge 1$, there exists $b_3\in B_1$ such that
\begin{equation}
\label{eq:ddd0}
|\alpha_2(b_3)-t_+/r_{\alpha_2}|\ll q^d (t_+/r_{\alpha_2})M^{-1}\ll q^d r_{\alpha_2}^{-1} M^{-1} ,
\end{equation}
and
\begin{equation}
\label{eq:b300}
\|b_3\|\le e^{O(|\log (t_+/r_{\alpha_2})|M^{\eta+1})}= e^{O(q^{d}M^{\eta+1})}.
\end{equation}
Since in the next step we will apply Lemma \ref{l:quant1} with $a=b_3$, $y=u_{\alpha_2}(r_{\alpha_2})y_0'$ and $y'=v_2 y$, we have to take
$b_3$ so that
\begin{equation}
\label{eq:b_3}
\|b_3\|< \|e-v_2\|^{-1/(2d)},
\end{equation}
To arrange this (see \eqref{eq:prod2}), we can take $b_3$ of size 
\begin{equation}
\label{eq:b_3_00}
\|b_3\|\le \omega_1 2^{(\delta-o(1)) k/(2d)}
\end{equation}
with sufficiently small $\omega_1>0$.
Moreover, in the next step, 
we will apply Lemma \ref{l:quant1} with $a=b_3b_2$, $y=h_2y_0''$ and $y'=v_1 y$.
Hence, we also have to arrange that
\begin{equation}
\label{eq:b_3_1}
\|b_3b_2\|< \|e-v_1\|^{-1/(2d)}.
\end{equation}
For this purpose, we can choose $b_3$ such that 
\begin{equation}
\label{eq:b_3_110}
\|b_3b_2\|\le \omega_2 2^{\zeta^{\ell-1} k/(2d)}
\end{equation}
with sufficiently small $\omega_2>0$ (see \eqref{eq:prod}).
We pick the parameter $\zeta$ so that 
\begin{equation}
\label{eq:zeta}
\zeta<1/(8c_1d).
\end{equation}
Then in view of \eqref{eq:b_20}, $\|b_2\|\le 2^{\zeta^{\ell-1} k/(4d)}$ for sufficiently large $T$. Hence, if take $b_3$ of size
\begin{equation}
\label{eq:b_3_11}
\|b_3\|\le \omega_2 2^{\zeta^{\ell-1} k/(4d)},
\end{equation}
then \eqref{eq:b_3_110} holds.
Now let us take $M$ such that 
\begin{equation}
\label{eq:m}
q^{d}M^{\eta+1}\le \delta' k\quad\hbox{ with $\delta'>0$.}
\end{equation}
If $\delta'$ is sufficiently small, 
then \eqref{eq:b300} implies that \eqref{eq:b_3_00} and \eqref{eq:b_3_11} hold,
and consequently \eqref{eq:b_3} and \eqref{eq:b_3_1} hold.
Since $M$ in \eqref{eq:m} can be chosen to satisfy $M\gg (k/q^{d})^{1/(\eta+1)}$, it follows from 
(\ref{eq:ddd0}) that
\begin{equation}
\label{eq:lll}
|\alpha_2(b_3)r_{\alpha_2}-t_+|\ll \frac{q^{d+d/(\eta+1)}}{k^{1/(\eta+1)}}\le
\theta,
\end{equation}
where $\theta:=\frac{L^{d(\eta+2)/(\eta+1)}}{k^{1/(\eta+1)}}$.
In order for this to give a non-trivial estimate, we have to require that
the parameter $L$ is chosen so that
\begin{equation}
\label{eq:L2}
L^{d(\eta+2)}=o(k)\quad \hbox{as $T\to \infty$,}
\end{equation}
which is a strengthening of our previous assumption \eqref{eq:L1}.
We assume that \eqref{eq:L2} holds. Then, in particular, $\theta\to 0$ as $T\to \infty$.

Now in \eqref{eq:eeee} we choose 
$\ve_1 =c_1\,\theta$ and $\ve_2 =c_2\,\theta$
with some $0<c_1<c_2$. 
Then since $u_{\alpha_2}(t_+)y_0''\in \mathcal{W}(\vartheta,\ve_1,\ve_2)$, there exists $z\in \Lambda_{y_0''}$ such that 
\begin{equation}
\label{eq:baba}
\|u_{\alpha_2}(t_+)z\|< \vartheta\quad\hbox{and}\quad  c_1\,\theta<|N(u_{\alpha_2}(t_+)z)|<c_2\,\theta.
\end{equation}
It follows from \eqref{eq:lll} that 
$$
u_{\alpha_2}(\alpha_2(b_3)r_{\alpha_2})z=u_{\alpha_2}(t_+)z+O(\theta \vartheta).
$$
Hence, we conclude that the vector $u_{\alpha_2}(\alpha_2(b_3)r_{\alpha_2})z$
also satisfies bounds of the form \eqref{eq:baba} (provided that the constants $c_1$ and $c_2$ are sufficiently large), Namely, we deduce that
$$
\|u_{\alpha_2}(\alpha_2(b_3)r_{\alpha_2})z\|< 2\vartheta
$$
if $T$ is sufficiently large, and 
$$
c^{(1)}_1\,\theta<|N(u_{\alpha_2}(\alpha_2(b_3)r_{\alpha_2})z)|<c^{(1)}_2\,\theta
$$
for some $c^{(1)}_1, c^{(1)}_2>0$.
Thus it follows that 
$$
b_3 u_{\alpha_2}(r_{\alpha_2})y_0''= u_{\alpha_2}(\alpha_2(b_3)r_{\alpha_2})y_0''\in \mathcal{W}(2\vartheta,c_1^{(1)}\,\theta,c_2^{(1)}\,\theta).
$$
Since \eqref{eq:b_3} hold, we can apply Lemma \ref{l:quant1} with $a=b_3$ and  $y=u_{\alpha_2}(r_{\alpha_2})y_0'$ to deduce that
$$
b_3b_2  h_2 y_0''=b_3  h_3 y_0''=b_3 v_2 u_{\alpha_2}(r_{\alpha_2})y_0''
\in \mathcal{W}(6\vartheta,c_1^{(2)}\,\theta,c_2^{(2)}\,\theta)
$$
for some $c_1^{(2)},c_2^{(2)}>0$. Since \eqref{eq:b_3_1} holds,
we can apply again Lemma \ref{l:quant1}  with  $a=b_3b_2$ and $y=h_2y_0''$ to conclude that
$$
(b_3b_2)b  h y_0''= (b_3b_2)  h_1 y_0''=(b_3b_2)  v_1h_2 y_0''\in 
\mathcal{W}(18\vartheta,c_1^{(3)}\,\theta,c_2^{(3)}\,\theta)
$$
for some $c_1^{(3)},c_2^{(3)}>0$.
Finally, combining this with \eqref{eq:first}, we deduce that
for $a:=a_0b_3 b_2 b$, we have
$$
ay=(b_3b_2)b  h y_0''\in 
\mathcal{W}(18\vartheta,c_1^{(3)}\,\theta,c_2^{(3)}\,\theta).
$$
We note that
$$
\|a\|\le \|a_0\|\, \|b_3 b_2\|\, \|b\|\ll e^{\nu T}\, 2^{\zeta^{\ell-1}k/(2d)}\, T^{O(1)}\le e^{(\nu+1)T}
$$
for sufficiently large $T$. 
This proves that 
\begin{equation}
\label{eq:case2}
A(e^{ (\nu+1)T})y\cap \mathcal{W}\left(18\vartheta, c_2^{(3)}\frac{L^{d(\eta+2)/(\eta+1)}}{k^{1/(\eta+1)}}\right)\ne \emptyset
\end{equation}
for sufficiently large $T$.

\vspace{0.4cm}

\noindent \underline{Case 3:} {\it $w$ is close to a torsion point with large period.}
We consider the set
$$
\mathcal{D}_{x_0}(k,L):=\{z\in V/\Delta_{x_0}:\, \|z-w_0\|<q^{-k}\hbox{ for some $w_0\in q^{-1}\Delta_{x_0}$ and $q\ge L$} \}.
$$
Let $\hbox{diam}^*(S)$ denote the supremum of diameters of the connected components of the set $S$.
We observe that 
\begin{equation}\label{eq:d1}
\hbox{diam}^*(\mathcal{D}_{x_0}(k,L))\le \sum_{q\ge L}\sum_{w_0\in q^{-1}\Delta_0} 2q^{-k}\ll \sum_{q\ge L} q^{d-k}\ll L^{d+1-k}.
\end{equation}
We recall that by \eqref{eq:asss} either $w=0$ or $\|w\|\ge T^{-\beta}$.
Hence, according to our assumption in Case 3, we must have $\|w\|\ge T^{-\beta}$.
Without loss of generality, let us assume that $|w_1|\ge T^{-\beta}$.
We consider the one-parameter subgroup $a(t):=\hbox{diag}(e^t,e^{-t},1,\ldots, 1)$ of $A$.
We observe that for $t\ge 0$,
$$
\|a(t)w-a(0)w\|\ge (e^t-1)|w_1|\ge t T^{-\beta}.
$$
Hence, 
\begin{equation}\label{eq:d2}
\hbox{diam}(a([0,\log (1+T^{-\beta})])w)\ge T^{-2\beta}/2.
\end{equation}
We choose the parameter $L$ so that 
\begin{equation}
\label{eq:L3}
L^{d+1-k}< \omega\, T^{-2\beta}
\end{equation}
with sufficiently small $\omega>0$.
Then comparing \eqref{eq:d1} and \eqref{eq:d2}, we deduce that there exists $a(t)$ with $\|a(t)-e\|\ll T^{-\beta}$ such that $a(t)w\notin \mathcal{D}_{x_0}(k,L)$. We replace \eqref{eq:start2} by
$$
a(t)a_0y=a(t)g y_0'
$$
where the point $y_0'$ corresponds to the grid $\Delta_{x_0}+ (a(t)g)^{-1} a(t)w$.
Hence, if we replace $a_0$ by $a(t)a_0$ and $g$ by $a(t)g$,
we obtain \eqref{eq:start2} with $w$ satisfying either the condition of Case 1 or the condition of Case 2. Hence, we can reduce the proof to the situations considered in Cases 1 or 2.
This reduction is possibly provided that \eqref{eq:L3} holds,
so that we can choose 
\begin{equation}
\label{eq:L4}
L\ll T^{2\beta/(k-d-1)}.
\end{equation}
Then \eqref{eq:case2} becomes
\begin{equation}
\label{eq:case2_1}
A(e^{ (\nu+1)T})y\cap \mathcal{W}\left(18 \vartheta, O\left(\frac{T^{\frac{2\beta d(\eta+2)}{(\eta+1)(k-d-1)}}}{k^{1/(\eta+1)}}\right)\right)\ne \emptyset.
\end{equation}

\vspace{0.5cm}

Finally, we complete the proof of the theorem by combining the estimates obtained in Cases 1 and 2, and optimising the parameter $k=k(T)$. We recall that $k$ is required to satisfy \eqref{eq:k1} with $R=e^{(\nu+1)T}$, so that $k=\frac{\log T}{\rho(T)}$ for some $\rho(T)\to \infty$. Moreover, $k$ has to satisfy \eqref{eq:L2}.
Since we are assuming that $k\to\infty$, it follows from \eqref{eq:L4} that \eqref{eq:L2} holds provided that 
$T^{\delta''/k}=o(k)$ with some $\delta''>2\beta d (\eta+3)$. Therefore,
the parameter $k$ has to be chosen so that
$$
\frac{\delta''}{k}\log T-\log k =\delta''\, \rho(T) -\log\log T+\log \rho(T) \to -\infty.
$$
Hence, we can take $\rho(T)=\kappa \log\log T$ for sufficiently small $\kappa>0$.
Then \eqref{eq:k1} holds with $s(R)\gg \log_{(4)} R$, where $R=e^{(\nu+1)T}$.
Hence, \eqref{eq:RRR1} implies that for sufficiently large $R$,
$$
A(R)y\cap \mathcal{W}(3,O((\log_{(5)} R	)^{-d\sigma}))\ne \emptyset.
$$
This proves the proposition in Case 1. 

Also with this choice of $\rho(T)=\kappa \log\log T$, provided that $\kappa$ is chosen sufficiently small,
we obtain that for some $\omega_0>0$,
\begin{align*}
\log\left(\frac{T^{\frac{2\beta d(\eta+2)}{(\eta+1)(k-d-1)}}}{k^{1/(\eta+1)}}\right)&\ll \omega_0\frac{\log T}{k}-\log k=\omega_0 \rho(T)-\log\log T+\log \rho(T)\\
&\ll -\log\log T.
\end{align*}
Hence, \eqref{eq:case2_1} implies that for some $\delta>0$,
$$
A(e^{(\nu+1)T})y\cap \mathcal{W}(18 \vartheta,O((\log T)^{-\delta}))\ne \emptyset,
$$ 
which proves the proposition in Case 2 and completes the proof of the theorem.
\end{proof}

\begin{proof}[Proof of theorem \ref{th:main0} ]
We consider the family of grids
$$
\Lambda(u,v,\alpha,\beta)\defeq\{{}^t(x,xu-y-\alpha,xv-z-\beta):\, x,y,z\in \ZZ\}.
$$
We note that the lattices $\Lambda(u,v, 0,0)$ with $(u,v)\in\RR^2$ are precisely
the lattices in the orbit $U\ZZ^3$, where $U$ is the expanding horospherical subgroup
defined in \eqref{eq:uuu}.
Hence, it follows from Proposition \ref{p:rec2} that for almost every $(u,v)\in \R^2$,
the lattice $\Lambda(u,v, 0,0)$ satisfies the assumption of Proposition \ref{p:supp}. 
Therefore, by this proposition, the grid
$\Lambda(u,v, \alpha,\beta)$ with arbitrary $\alpha,\beta\in \R$ has property (WR)
with $h(T)= (\log_{(5)} T)^{\delta}$. Thus, by Proposition \ref{p:dyn},
the grid $\Lambda(u,v, \alpha,\beta)$ is $h(T)$-multiplicatively approximable.
This implies that there exists a sequence $v_n={}^t(q_n, q_n u-r_n-\alpha, q_n v -s_n-\beta)$ with $q_n,r_n,s_n\in\ZZ$ such that
$v_n\to\infty$, and 
$$
0< (\log_{(5)} \|v_n\|)^{\delta} |q_n| |q_n u-r_n-\alpha||q_n v-s_n-\beta|<1.
$$
In particular, it follows that $0 \ne |q_n| |q_n u-r_n-\alpha||q_n v-s_n-\beta|\to 0$.
Since $r_n,s_n\in\ZZ$, this can only happen if $|q_n|\to\infty$. We observe that there exist $c_1,c_2>0$ such that
$\|v_n\|\ge c_1|q_n|-c_2$, so that 
$$
(\log_{(5)} (c_1|q_n|-c_2))^{\delta} |q_n| \left<q_n u-\alpha\right>
\left<q_n v-\beta\right>\le 
(\log_{(5)} \|v_n\|)^{\delta} |q_n| |q_n u-r_n-\alpha||q_n v-s_n-\beta|.
$$
Hence, we deduce that 
$$
\liminf_{|q|\to\infty} (\log_{(5)} |q|)^{\delta} |q| \left<q u-\alpha\right>
\left<q v-\beta\right>\le 1,
$$
and Theorem \ref{th:main0} follows.
\end{proof}

\begin{proof}[Proof of theorem \ref{th:main}]
The proof of Theorem \ref{th:main} is similar.
We note that it is sufficient to prove the theorem for almost every unimodular lattice $\Delta$
because general lattices can be obtained by rescaling.
By Proposition \ref{p:rec1}, almost every $x\in X$ satisfies the 
assumptions of Proposition \ref{p:supp}.
Let $\Delta$ denote the lattice corresponding to $x$ for which Proposition \ref{p:supp} applies.
Then combining Proposition \ref{p:supp} and Proposition \ref{p:dyn},
we deduce that for every $w\in \R^d$, the grid $\Delta+w$
is $h(T)$-multiplicatively approximable with $h(T)=(\log_{(5)} T)^{\delta}$.
This completes the proof of the theorem.
\end{proof}

\end{document}